\documentclass{article}
\usepackage{amssymb,amsmath,graphicx,ucs,hyperref}
\usepackage[utf8x]{inputenc}
\usepackage[english]{babel}

\newtheorem{corollary}{Corollary}
\newtheorem{theorem}{Theorem}
\newtheorem{proposition}{Proposition}
\newtheorem{lemma}{Lemma}
\newtheorem{definition}{Definition}

\newenvironment{proof}{\noindent \emph{Proof. }}{\hfill \hbox{\rlap{$\sqcap$}$\sqcup$}\\}

\title{Canonical projection tilings defined by patterns
  \thanks{
  Part of this work has been done in the Departamento de Ingenier\'ia Matem\'atica and Centro de Modelamiento Matem\'atico, CNRS-UMI 2807, Universitad de Chile, Santiago, Chile.}
}

\author{
Nicolas Bédaride
\footnote{Aix Marseille Univ., CNRS, Centrale Marseille, I2M, UMR 7373, 13453 Marseille, France.}
\and
Thomas Fernique
\footnote{Univ. Paris 13, CNRS, Sorbonne Paris Cité, UMR 7030, 93430 Villetaneuse, France.}
}
\date{}

\begin{document}

\maketitle

\begin{abstract}
  We give a necessary and sufficient condition on a $d$-dimensional affine subspace of $\mathbb{R}^n$ to be characterized by a finite set of patterns which are forbidden to appear in its digitization.
  This can also be stated in terms of local rules for canonical projection tilings, or subshift of finite type.
  This provides a link between algebraic properties of affine subspaces and combinatorics of their digitizations.
  The condition relies on the notion of {\em coincidence} and can be effectively checked.
  As a corollary, we get that only algebraic subspaces can be characterized by patterns.
\end{abstract}

\section{Introduction}
\label{sec:introduction}

The cut and project scheme is a popular way to define aperiodic tilings (see, {\em e.g.}, \cite{BG13} and references therein).
A rich subfamily of these tilings is formed by the so-called {\em canonical projection tilings}, which are digitizations of affines $d$-planes of $\mathbb{R}^n$.
It includes, {\em e.g.}, Sturmian words (lines of $\mathbb{R}^2$), billiard words (lines of $\mathbb{R}^3$), Ammann-Beenker tilings ($2$-planes of $\mathbb{R}^4$), Penrose tilings ($2$-planes of $\mathbb{R}^5$) or icosahedral tilings ($3$-planes of $\mathbb{R}^6$).\\

In particular, canonical projection tilings with irrational {\em slopes} (the $d$-plane it digitizes) are widely used in condensed matter theory to model {\em quasicrystals}.
Both are indeed aperiodic but nonetheless ``ordered'' (in a sense that can be slightly different in condensed matter theory or mathematics).
In this context, assuming that the stability of a real material is governed only by finite range energetic interactions, it is important to decide whether such a tiling can be characterized only by its patterns of a given (finite) size - one speaks about {\em local rules}.
This issue has be tackled by numerous authors and several conditions have been obtained \cite{Bur88, Kat88, Lev88, Soc90, LPS92, Kat95, LP95, Le95, Le97, BF13, BF15a, BF15b, BF17}, but not complete characterization yet exists (except if we allow tiles to be {\em decorated}, see \cite{FS18}, but the situation becomes quite different).\\

More precisely, let $G'(n,d)$ denote the $d$-planes of $\mathbb{R}^n$ which are {\em generic}, {\em i.e.}, not included in a strict rational subspace of $\mathbb{R}^n$.
A slope $E\in G'(n,d)$ is said to be {\em characterized by patterns} if there is a finite set of (finite) patterns, called {\em forbidden patterns}, such that any canonical projection tiling with a slope in $G'(n,d)$ which does not contain any of these patterns has a slope parallel to $E$.
It is generally unclear to determine whether a given slope is characterized by patterns.\\

We here provide an equivalent characterization which reduces to decide whether a system of polynomial equations has a finite number of solutions and can thus be effectively checked.
It relies on the geometric notion of {\em coincidence}, first introduced in \cite{BF15a}.
A slope $E\in G'(n,d)$ is said to be characterized by coincidences if it is the only slope in $G'(n,d)$ which admits all these coincidences.
Our main result is that patterns and coincidences are equivalent:

\begin{theorem}\label{th:main}
  A slope in $G'(n,d)$ is characterized by patterns if and only if it is characterized by coincidences.
\end{theorem}

As far as we know, this is the first necessary and sufficient condition for local rules for planar tilings with slope in $G(n,d)$.
It is moreover easily checked with computer algebra whether a given slope satisfied this condition.
However, we have to acknowledge that it relies on two major assumptions.\\

The first assumption is that only the set of {\em generic} slopes is considered, whereas we would like to have a characterization which considers any possible slope.
Indeed, maybe the patterns which characterize such a slope could allow a planar tiling with a non-generic slope (that is, in $G(n,d)\backslash G'(n,d)$), although we have no such example.
We discuss this more in details in Section~\ref{sec:penrose} on the Penrose tilings.\\

The second assumption is that only {\em planar} tilings are considered, whereas we would like to have a characterization which considers any tiling.
Indeed, maybe a set of patterns which allow only planar tilings with a given slope and no other planar tilings nevertheless allow a non-planar tiling, that is, the digitization of some non-flat surface (the case of a surface which stays at bounded but possibly large distance from a plane is somehow intermediary - this corresponds to the issue of {\em weak local rules} raised in \cite{Lev88}).
In \cite{BF17}, we tackled this issue for $2$-planes in $\mathbb{R}^4$ and obtained a characterization which shows that, among the polynomial equations that the coordinates of the $2$-planes must satisfy, at least $4$ of them must be "sufficiently simple", namely, linear (a corollary is that slopes are always based on at most quadratic irrationalities in this case).
A similar approach would yield a lower bound on the number of linear equations for a $d$-plane in $\mathbb{R}^n$ (namely $\tfrac{dn}{d+1}$): this allows to bound the maximal algebraic degree of irrationalities defining the slopes, but this necessary condition is likely not sufficient.
A sufficient condition is provided in \cite{BF15b}, but it is also likely not tight.
The issue of characterizing {\em planarity} by patterns is thus still open.\\

The rest of the paper is organized as follows.
Section~\ref{sec:tilings} defines canonical projection tilings and the (Grassmann) coordinates of their slopes as well as patterns and tools to study them.
These notions are also useful to study the statistics of patterns in such tilings \cite{HKWS16,HJKW18,Jul10}.\\

Section~\ref{sec:coincidences} introduces coincidences and the associated polynomial equations.
The two next sections prove Theorem~\ref{th:main}: Section~\ref{sec:patterns2coincidences} shows that if a slope in $G'(n,d)$ is characterized by patterns, then it is characterized by coincidences, and Section~\ref{sec:coincidences2patterns} shows the converse.
The three last sections illustrates the theorem (and its limits) on several examples: a "typical" case (Section~\ref{sec:typical}, where everything works fine), the Ammann-Beenker tilings (Section~\ref{sec:ammann_beenker}, which illustrates the case of a slope not characterized by its coincidences) and the Penrose tilings (Section~\ref{sec:penrose}, which illustrate the problem with non-generic slopes).

\section{Canonical projection tilings and patterns}
\label{sec:tilings}

Let $\vec{v}_1,\ldots,\vec{v}_n$ be vectors of $\mathbb{R}^d$, $n>d$, such that any $d$ of them are independent.
For $0<i_1<\ldots< i_d\leq n$, the {\em prototile} $T_{i_1,\ldots,i_d}$ is the non-empty interior parallelotope defined by
$$
T_{i_1,\ldots,i_d}=\{\lambda_{i_1}\vec{v}_{i_1}+\ldots+\lambda_{i_d}\vec{v}_{i_d}~|~0\leq \lambda_{i_1},\ldots,\lambda_{i_d}\leq 1\}.
$$
A {\em tile} is a translated prototile.
An {\em $n\to d$ canonical tiling} is a covering of $\mathbb{R}^d$ by interior-disjoint tiles.
We moreover assume that the tiling is {\em face-to-face}, that is, whenever two tiles intersect, their intersection is a full face (of any dimension) of each tile.\\

The {\em lift} of an $n\to d$ canonical tiling is a $d$-dimensional ``stepped'' manifold of $\mathbb{R}^n$ obtained as follows.
Let $\vec{e}_1,\ldots,\vec{e}_n$ denote the standard basis of $\mathbb{R}^n$.
First, an arbitary vertex of the tiling is mapped onto the origin of $\mathbb{R}^n$.
Then, each translate of $T_{i_1,\ldots,i_d}$ is mapped onto a unit face of $\mathbb{Z}^n$ generated by $\vec{e}_{i_1},\ldots,\vec{e}_{i_d}$, translated such that whenever two tiles share an edge $\vec{v}_i$, their images share an edge $\vec{e}_i$ (this is a consistent definition because any closed path on the tiling is mapped onto a closed path in $\mathbb{R}^n$).
This is nothing but a generalization of the natural interpretation of a rhombille tiling as a surface in the $3$-dimensional space.\\

A $n\to d$ canonical tiling is said to be {\em planar} if there is a $d$-dimensional affine subspace $E$ of $\mathbb{R}^n$ such that the ``tube'' $E+[0,1]^n$ contains a lift of the tiling.
As a discrete object (vertices in $\mathbb{Z}^n$) which stay close to the affine subspace $E$, the tiling can be seen as a {\em digitization} of $E$.
The $d$-plane $E$ is called the {\em slope} of the tiling (it is uniquely defined up to a translation).
Figures \ref{fig:typical_tiling}, \ref{fig:beenker_tiling1}, \ref{fig:beenker_tiling2} and \ref{fig:penrose_tiling} (pages~\pageref{fig:typical_tiling} to \pageref{fig:penrose_tiling}) depict some planar canonical tilings.
The set of $d$-planes of $\mathbb{R}^n$, called {\em Grassmannian}, shall be here denoted by $G(n,d)$.\\

Planar canonical tilings are also called {\em canonical projection tilings}.
Indeed, a planar canonical tiling of slope $E$ can be obtained as follows.
Let $E'$ denote a complementary space of $E$ ({\em e.g.}, the orthogonal space).
Let $\pi$ and $\pi'$ denote the orthogonal projections onto, respectively, $E$ and $E'$.
Let $W_E$ denote the projection of a hypercube $\vec{x}+[0,1]^n$ for some $\vec{x}\in E$.
The polytope $W_E$ is called the {\em window} of the tiling.
By selecting the unit $d$-dim. faces of $\mathbb{Z}^n$ which project under $\pi'$ inside $W_E$ and projecting them under $\pi$ onto $E$ we get the planar canonical tiling of slope $E$ (see, {\em e.g.}, \cite{BG13} for a fully detailed presentation).\\

We shall call {\em generic} a $d$-plane $E$ of $\mathbb{R}^n$ such that the only rational subspace of $\mathbb{R}^n$ which contains $E$ is $\mathbb{R}^n$ itself.
This is equivalent to say that the projection under $\pi'$ of $\mathbb{Z}^n$ is dense in $E^\bot$, or that the projection under $\pi'$ of the vertices of the planar tiling of slope $E$ are dense in the window $W_E$.
We denote by $G'(n,d)$ the set of generic slopes: this is a dense open subset of the Grassmannian $G(n,d)$.\\

{\em Grassmann coordinates} provide a convenient way to describe slopes (see, {\em e.g.}, \cite{HP94} for a detailed account on Grassmann coordinates).
The Grassmann coordinates of a $d$-plane $E$ of $\mathbb{R}^n$ are the $d\times d$ minors of a matrix $M$ whose columns generate $E$.
They are unique up to a common multiplicative constant.
We shall denote by $G_{i_1\cdots i_d}$ the minor obtained from the lines $i_1,\ldots,i_d$.
The Grassmann coordinates of $E$ characterize it and are independent of the matrix $M$ up to a normalization.
One checks that the matrix $M$ defined by $M_{ij}:=G_{1,2,...,i−1,j,i+1,...,d}$ is a possible choice of generators of $E$.\\

Since a $d$-plane of $\mathbb{R}^n$ has $\binom{n}{d}$ Grassmann coordinates but the Grassmannian $G(n,d)$ has dimension $d(n-d)$, there must be relation between Grassmann coordinates.
Indeed, any two coordinates $G_{i_1\cdots i_d}$ and $G_{j_1\cdots j_d}$ satisfy, for any $1\leq k\leq d$, the so-called {\em Plücker relations}
$$
G_{i_1\cdots i_d}G_{j_1\cdots j_d}=\sum_{l=1}^d\underbrace{G_{i_1\cdots i_d}G_{j_1\cdots j_d}}_{\textrm{swap $i_k$ and $j_l$}},
$$
where, by convention, $G_{i_1\cdots i_d}$ is equal to zero if two indices are equal, and has opposite sign if two indices are permuted.
Conversely, any non-zero $\binom{n}{d}$-tuple $(G_{i_1\cdots i_d})$ of real numbers satisfying all these quadratic relations are the Grassmann coordinates of some $d$-plane $E$ of $\mathbb{R}^n$.\\

The Grassmann coordinates of a $d$-plane $E$ can actually be ``seen'' on the planar canonical tiling of slope $E$.
With the normalization $||(G_{i_1\cdots i_d})||_1=1$, $|G_{i_1\cdots i_d}|$ indeed gives the frequency of the tile $T_{i_1,\ldots,i_d}$ in the tiling.
Moreover, the sign of $G_{i_1\cdots i_d}$ is equal to the sign of $\det(\vec{v}_{i_1},\ldots,\vec{v}_{i_d})$, that is, once the $\vec{v}_i$'s which define the tiles are fixed, the sign of a Grassmann coordinate is the same in the slopes of all the planar canonical tilings (indeed, if a non-zero Grassmann coordinate would have a different sign in the slopes of two planar tilings, then a continuous transformation from one to another would go through a slope where this Grassmann coordinate is zero; this would correspond to a tile whose volume is equal to zero and changes its sign, but we consider only tiles with positive volume).
This restricts the set of $d$-planes of $\mathbb{R}^n$ that can be achieved as a slope of a planar tiling.\\

We call {\em pattern} a finite connected subset of the edges of a canonical tiling.
It can be lifted to $\mathbb{R}^n$ as done for canonical tilings.
As stated in the introduction, a slope $E\in G'(n,d)$ is said to be {\em characterized by patterns} if there is a finite set of patterns, called {\em forbidden patterns}, such that any canonical projection tiling with a slope in $G'(n,d)$ which does not contain any of these patterns has a slope parallel to $E$.\\

Among patterns, those which appear in planar tilings can be characterized via the window (Fig.~\ref{fig:pattern_region}):

\begin{proposition}
\label{prop:pattern_region}
  Let $E$ be a $d$-plane of $\mathbb{R}^n$ and $W_E$ be the associated window.
  The {\em region} of a lifted pattern $\widehat{P}$ is the convex polytope defined by
  $$
  R_E(\widehat{P}):=\bigcap\{(W_E-\pi'\vec{x})~|~\vec{x}\in\mathbb{Z}^n,~\vec{x}\in\widehat{P}\}.
  $$
  Then, for any $\vec{y}\in\mathbb{Z}^n$, the pattern $\pi'(\vec{y}+\widehat{P})$ appears in the planar tiling of slope $E$ if and only if $\pi'\vec{y}$ belongs to the region of $\widehat{P}$.
\end{proposition}

\begin{proof}
  Let $\vec{x}\in\mathbb{Z}^n$ and $P$ a pattern which can be lifted on $\widehat{P}$.
  First, assume that $P$ is formed by a unique edge which connects the points $\pi\vec{a}$ and $\pi\vec{b}$.
  By definition of a canonical projection tiling, the edge $\pi\vec{x}+P$, which connects $\pi\vec{x}+\pi\vec{a}$ to $\pi\vec{x}+\pi\vec{b}$, appears in the tiling if and only if the edge connecting $\vec{x}+\vec{a}$ to $\vec{x}+\vec{b}$ lies into the tube $E+[0,1]^n$, that is, if and only if the edge connecting $\pi'\vec{x}+\pi'\vec{a}$ to $\pi'\vec{x}+\pi'\vec{b}$ lies into $W_E$.
  This happens exactly when $\pi'\vec{x}$ belongs to $(W_E-\pi'\vec{a})\cap (W_E-\pi'\vec{b})$.
  This extends to patterns with more edges by iterating the same argument edge by edge.
\end{proof}

\begin{figure}
  \includegraphics[width=\textwidth]{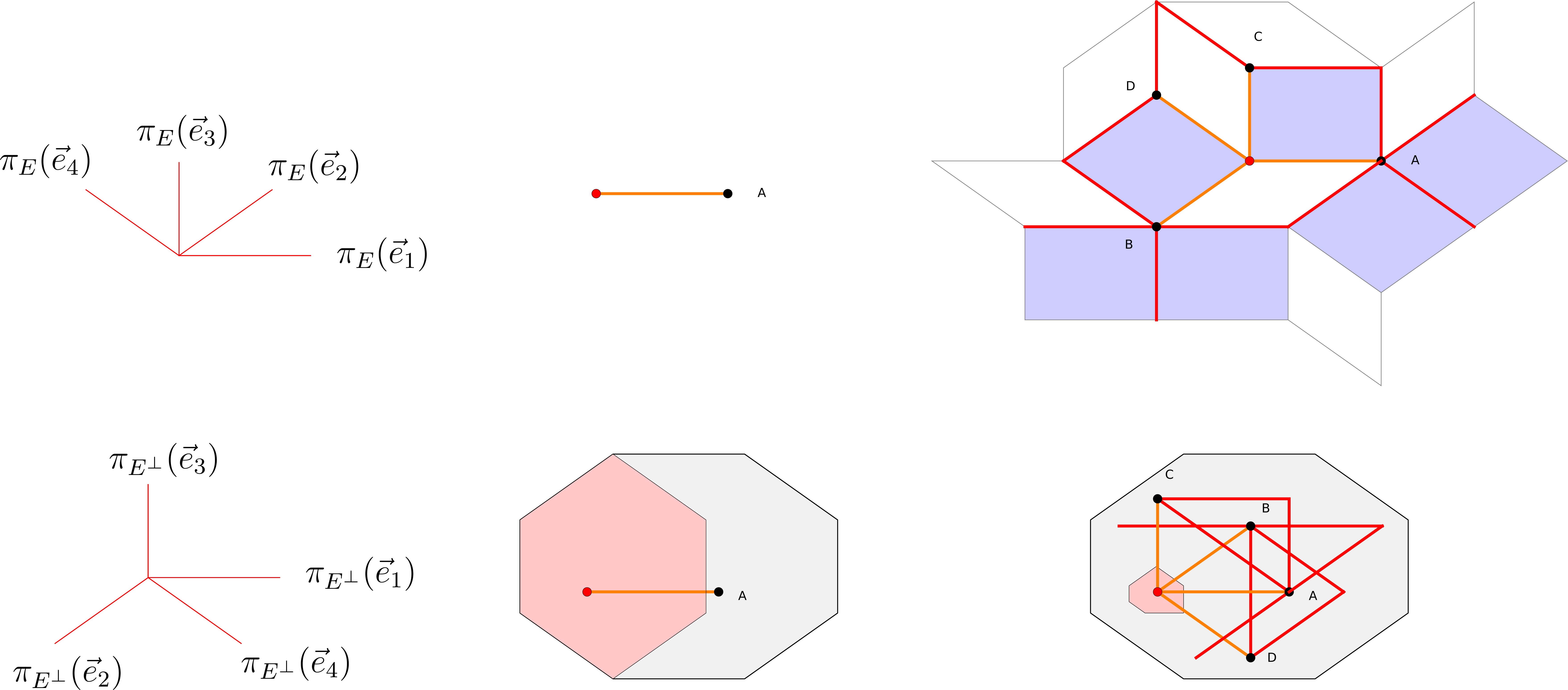}
  \caption{
  Top: projection on a slope $E$ of the basis vectors of $\mathbb{R}^4$, a pattern formed by a unique edge and a more complicated pattern (from left to right).
  Bottom: the same objects projected in the (octagonal) window, where the darkest polygons depict the regions associated with each pattern.
}
  \label{fig:pattern_region}
\end{figure}

In particular, if $E$ is generic, then the density of $\pi'\mathbb{Z}^n$ in $E'$ ensures that any pattern whose region has non-empty interior in $W_E$ appears in the planar tiling with slope $E$.
This shall play a key role in the proof of Th.~\ref{th:main}.

\section{Coincidences}
\label{sec:coincidences}

We shall now introduce the main notion of this paper, namely {\em coincidences}:

\begin{definition}
  \label{def:coincidence}
  A {\em coincidence} of a $d$-plane $E\subset\mathbb{R}^n$ is a set of $n-d+1$ pairwise non-parallel unit $n-d-1$ dimensional faces of $\mathbb{Z}^n$ whose projections under $\pi'$ are concurrent in the window $W_E$.
\end{definition}

When there is no ambiguity, we shall also call coincidence the $n-d+1$ points of $\mathbb{R}^n$ which project onto the same point in $W_E$, or the projections in $W_E$ of the $n-d+1$ unit faces of the coincidence.
As stated in the introduction, a slope $E\in G'(n,d)$ is said to be {\em characterized by coincidences} if it is the only slope in $G'(n,d)$ which admits all these coincidences.
Theorem~\ref{th:main} states the equivalent power of patterns and coincidences to characterize slopes of $G'(n,d)$.
We shall prove it in the following sections.
Here, we first show that coincidences correspond to polynomial equations that can be effectively checked.

\begin{proposition}\label{prop:eq_coincidence1}
  A coincidence of a $d$-plane of $\mathbb{R}^n$ corresponds to a rational polynomial equation on its Grassmann coordinates.
\end{proposition}

\begin{proof}
  Let $\mathbb{Q}(G_{i_1,\ldots,i_d})$ denote the Grassmann coordinates rational fraction field.
  Let $\vec{u}_1,\ldots,\vec{u}_d$ be a basis of $E$ whose entries are in $\mathbb{Q}(G_{i_1,\ldots,i_d})$.
  For example, one can take $\vec{u}_{ij}=G_{1,\ldots,i-1,j,i+1,\ldots,d}$.
  Consider a coincidence, {\em i.e.}, points $\vec{x}_0,\ldots, \vec{x}_{n-d}$ of $\mathbb{R}^n$, each with (at most) $n-d-1$ non-integer entries, which project under $\pi'$ onto the same point of the window.
  There are thus, for $i=1,\ldots,n-d$, coefficients $\lambda_{i1},\ldots,\lambda_{id}$ such that
  $$
  \vec{x}_0-\vec{x}_i+\sum_{k=1}^d \lambda_{ik}\vec{u}_k=0.
  $$
  This can be seen as a system of $n(n-d)$ linear equations over $\mathbb{Q}(G_{i_1,\ldots,i_d})$ whose variables are the $d(n-d)$ coefficients $\lambda_{ik}$ and the $(n-d+1)(n-d-1)$ non-integer entries of the $\vec{x}_i$'s.
  The total number of variables is
  $$
  d(n-d)+(n-d+1)(n-d-1)=d(n-d)+(n-d)^2-1=n(n-d)-1,
  $$
  that is, just one less than the number of equations.
  The system is thus overdetermined: the nullity of its (polynomial) determinant yields the equation.
\end{proof}

\noindent Let us be more precise:

\begin{proposition}\label{prop:eq_coincidence2}
The rational polynomial equation corresponding to a coincidence of a $d$-plane of $\mathbb{R}^n$ is homogeneous of degree $n-d$.
\end{proposition}

\begin{proof}
  Let us rewrite the system of the previous proof in terms of matrices:
  $$
  \left(\begin{array}{cccccc}
    \vec{b}_1 & A_1 & U & 0 & \cdots & 0 \\
    \vec{b}_2 & A_2 & 0 & \ddots & \ddots & \vdots\\
    \vdots & \vdots & \vdots & \ddots & \ddots & 0\\
   \vec{b}_{n-d} &A_{n-d} & 0 & \cdots & 0 & U
\end{array}\right)
\left(\begin{array}{c}
  1\\
  r_1\\
  \vdots\\
  r_p\\
  \lambda_{11}\\
  \lambda_{12}\\
  \vdots\\
  \lambda_{n-d,d}
\end{array}\right)
=
0,
$$
where
\begin{itemize}
\item $p:=(n-d-1)(n-d+1)$;
\item $r_1,\ldots,r_p$ denote the non-integer entries of the $\vec{x}_i$'s;
\item $\vec{b}_i\in\mathbb{Z}^n$ tracks the integer entries, {\em i.e.}, the $j$-th entry of $\vec{b}_i$ is the constant term (without $r_i$) of the $j$-th entry of $\vec{x}_0-\vec{x}_i$;
\item $A_i$ is the $n\times p$ matrix tracking the $r_i$'s, {\em i.e.}, $(A_i)_{jk}$ equals $1$ (resp. $-1$) if the $j$-th entry of $\vec{x}_0$ (resp. $\vec{x}_i$) is $r_k$, or $0$ otherwise;
\item $U$ is the $n\times d$ matrix whose $i$-th column is $\vec{u}_i$.
\end{itemize}
Let $M$ denote the above (large) matrix.
This is a square matrix of size $n(n-d)$.
The $\lambda_{ij}$'s are not all equal to zero because the $\vec{x}_i$'s are distincts.
The determinant of $M$ is thus zero: this is the coincidence equation.
Let us compute it by blocks:
$$
\det(M)=\sum_{\varphi\in\Phi}\varepsilon(\varphi)D_\varphi D_{\overline{\varphi}},
$$
where
\begin{itemize}
\item $\Phi$ is the set of increasing maps from $\{1,\ldots,p+1\}$ to $\{1,\ldots,n(n-d)\}$;
\item $\varepsilon(\varphi)$ is the signature of the unique permutation of $\{1,\ldots,n(n-d)\}$ which extends $\varphi$ and is increasing on $\{p+2,\ldots,n(n-d)\}$; 
\item $D_\varphi$ is the determinant of the submatrix of $M$ obtained by keeping the columns $1,\ldots,p+1$ (the $\vec{b}_i$'s and $A_i$'s) and the rows $\varphi(1),\ldots,\varphi(p+1)$;
\item $D_{\overline{\varphi}}$ is the determinant of the submatrix of $M$ obtained by keeping the other columns and rows, that is, the columns $p+2,\ldots,n(n-d)$ (the $U$'s) and the rows whose indices are not in the image of $\varphi$.
\end{itemize}
For any $\varphi\in\Phi$, $D_\varphi$ is an integer as the determinant of a matrix whose entries are entries of the $\vec{b}_i$'s and $A_i$'s, hence integer.
Consider now $D_{\overline{\varphi}}$.
It is obtained by picking $d(n-d)$ rows of the $n(n-d)\times d(n-d)$ matrix with blocks $U$ on its diagonal.
If $\varphi$ does not pick exactly $d$ rows in each block $U$, then there is a block with $k>d$ selected rows.
Each permutation of $d(n-d)$ which appears in the computation of $D_{\overline{\varphi}}$ will then pick at least one coefficient outside this block $U$ (which has only $d$ column).
Since such a coefficient is always zero, $D_{\overline{\varphi}}$ is also zero.
Hence, the only non-zero $D_{\overline{\varphi}}$ are these for which $\varphi$ picks exactly $d$ rows within each block $U$.
They are determinants of a block diagonal matrix, where each of the $n-d$ blocks $d\times d$ yields a Grassmann coordinate.
The coincidence equation $\det(M)=0$ is thus a homogeneous equation of degree $n-d$.
\end{proof}

We illustrate in Section~\ref{sec:typical} how to find these equations for a given slope and deduce the corresponding coincidences.

\section{From patterns to coincidences}
\label{sec:patterns2coincidences}

We shall here prove the first part of Theorem~\ref{th:main}.

\begin{proposition}\label{prop:patterns2coincidences}
  If a slope in $G'(n,d)$ is characterized by patterns, then it is characterized by coincidences.
\end{proposition}

\begin{proof}
  Let $E\in G'(n,d)$ characterized by a finite set of forbidden patterns.
  Because of the continuity of $E\to R_E(\widehat{P})$, there is a neighborhood $\mathcal{V}$ of $E$ such that any forbidden pattern which has an empty region for $E$ still has an empty region for planes in $\mathcal{V}$, thus still does not appear in the planar tilings with slope in $\mathcal{V}$.  
  Consider a pattern whose region $R$ in $E$ is not empty.
  This region must have empty interior, otherwise the density of $\pi'\mathbb{Z}^n$ in $E'$, due to the genericity of $E$, ensures that $R$ would contain the projection of an integer point and thus the pattern would appear in $E$.\\
  Now, assume that $E$ is not characterized by coincidences and let us get a contradiction.
  By definition, there is a slope $F\in G'(n,d)$, $F\neq E$, with the same coincidences as $E$.
  Any region of a forbidden pattern, which has empty interior in $W_E$, still has empty interior in $W_F$, because its extremal points are coincidences (indeed, at least $n-d+1$ pairwise non-parallel half-spaces are necessary to define an empty interior polytope in a $n-d$-dim. space).
  Up to a translation of $F$, one can assume that $\pi'\mathbb{Z}^n$ has no point in the union of these region (which has empty interior).
  This ensures that none of the forbidden patterns appears in the planar tiling with slope $F$.
  Since $F\neq E$, this contradicts the hypothesis $E$ is characterized by patterns.
  Thus $E$ must be characterized by coincidences.
\end{proof}

Given a slope $E\in G'(n,d)$, we can compute the coincidence equations (previous section) and check whether $E$ is the only slope in $G'(n,d)$ to satisfy these equations.
If not, then the previous proposition ensures that planar canonical tilings with such a slope are not characterized by patterns.
We shall illustrate this with Ammann-Beenker tilings in Section~\ref{sec:ammann_beenker}.\\

Since the coincidences of a slope correspond to algebraic equations on the Grassmann coordinates of this slope, we get as a corollary the following result, first obtained in \cite{Le97}:

\begin{corollary}
  Any slope in $G'(n,d)$ characterized by patterns is algebraic.
\end{corollary}

\section{From coincidences to patterns}
\label{sec:coincidences2patterns}

We shall here prove the second part of Theorem~\ref{th:main}.
Given a slope $E\in G'(n,d)$ characterized by coincidences, we shall provide an effective way to find patterns which also characterize $E$.
The point is that when the slopes varies from $E$ to some $F$, a coincidence may {\em break} in $F$, {\em i.e.}, the $n-d+1$ pairwise non-parallel unit $n-d-1$-dim. faces of $\mathbb{Z}^n$ faces whose projections under $\pi'$ were concurrent in the window $W_E$ are no more concurrent.
We have to show that this creates a new region for a pattern which appears in the planar tiling of slope $F$ but did not appear in the planar tiling of slope $E$: this shall yield the pattern to forbid.
We make cases, depending on ``how much'' the coincidence is broken.\\

Let us first refine the notions of coincidences and patterns by introducing an integer parameter $r$.
An {\em $r$-coincidence} is a coincidence such that $r$ bounds the absolute values of the entries of the vertices of the faces involved in the coincidence, and an {\em $r$-pattern} a pattern of a canonical tiling obtained by choosing an arbitrary vertex of the tiling (the {\em center}) and all the vertices within distance $r+1$ from it, then the tiles determined by these vertices (Fig.~\ref{fig:r_patterns_window}, left).

\begin{proposition}
  Let $r\geq 0$ be an integer and $E$ be a $d$-plane of $\mathbb{R}^n$.
  Then, the regions of all the possible $r$-patterns form a partition of the window $W_E$.
\end{proposition}

\begin{proof}
The $r$-pattern whose region contains a given point $\vec{x}$ in the window is indeed determined as follows.
First, consider the union $P'$ of all the paths made of $r+1$ edges $\pm\pi'\vec{e}_i$ which start from this point and stay in the window (there is such paths since from any point $\vec{y}$ of the window and any $i$, either $\vec{y}+\pi'\vec{e}_i$ or $\vec{y}-\pi'\vec{e}_i$ belongs to the window).
Then, lift $P'$ onto a set $\widehat{P}$ of unit edges of $\mathbb{Z}^n$ (that is, $\widehat{P}$ is a connected set of edges which projects onto $P'$).
The projection of $\widehat{P}$ onto the slope $E$ yields an $r$-pattern whose region contains $\vec{x}$.
\end{proof}

\begin{figure}[hbtp]
  \centering
  \includegraphics[width=\textwidth]{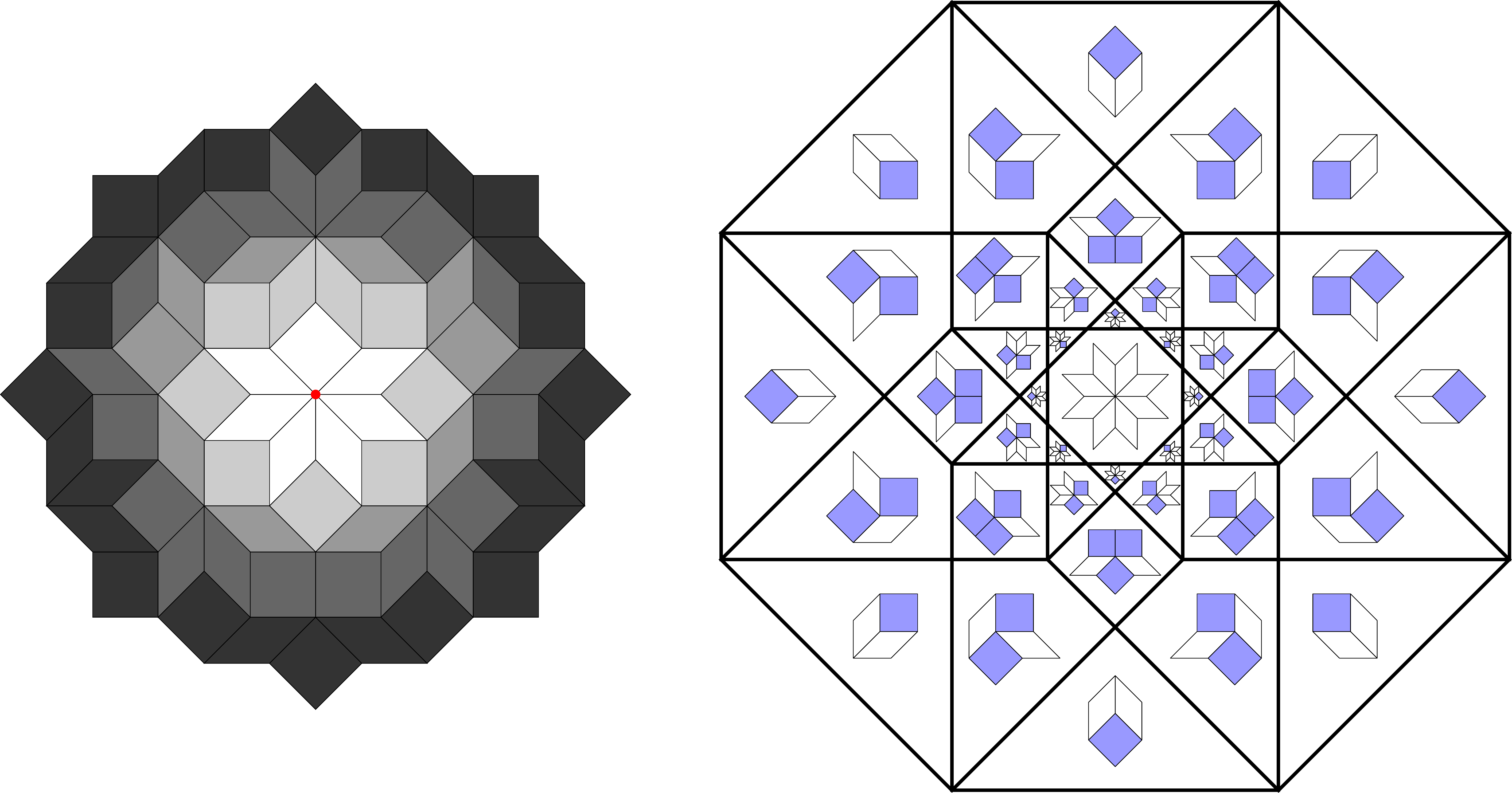}
  \caption{Concentric $r$-patterns with $r$ ranging from $0$ to $4$ (left) and partition of the window by $0$-patterns for a $4\to 2$ planar tiling (namely an Ammann-Beenker tiling, see Sec.~\ref{sec:ammann_beenker}), with each $0$-pattern being depicted inside its region.
  A point on the boundary between two regions belong to the region whose associated pattern has the more edges for inclusion (there is always such an inclusion).
  }
  \label{fig:r_patterns_window}
\end{figure}

Prop.~\ref{prop:pattern_region} moreover shows that the region of an $r$-pattern is the intersection of translations of $W_E$ by the projection under $\pi'$ of integer vectors whose entries have absolute values bounded by $r$.
Fig.~\ref{fig:pattern_region}. and Fig.~\ref{fig:r_patterns_window} (right) illustrates this.\\

\noindent The following lemma addresses ``small breaks'' of coincidences:

\begin{lemma}\label{lem:small_break}
  Let $E$ be a $d$-plane of $\mathbb{R}^n$.
  If an $r$-coincidence breaks in $F\in G'(n,d)$ but its faces still intersect pairwise, then any planar tiling of slope $F$ has an $r$-pattern which does not appear in some planar tiling with a slope parallel to $E$.
\end{lemma}

\begin{proof}
  Consider such a coincidence, {\em i.e.}, $n-d-1$-dim. unit faces of $\mathbb{Z}^n$ whose projections under $\pi'$ are concurrent in $W_E$.
  In $W_F$, the projections of these faces still pairwise intersect and thus define the boundary of a $n-d$-dimensional simplex $R$.
  Since they are translations of the faces of $W_F$ by integer vectors whose entries have absolute values bounded by $r$, this simplex $R$ contains the region of some $r$-pattern.
  This $r$-pattern appears in $F$ because $F$ is generic.\footnote{It could be false for a non-generic $F$ although we have no counter-example.}
  In $W_E$, the region of this $r$-pattern is a point (the coincidence) and, up to an eventual shift of $E$, no point of $\mathbb{Z}^n$ projects onto it, that is, the $r$-pattern does not appear in some planar tiling with a slope parallel to $E$.
\end{proof}

We shall now consider ``big breaks'', {\em i.e.}, when the slope modification is such that the faces of the coincidences do not anymore pairwise intersect.
Intuitively, it should be easier to find a new pattern after such a break than after a small one, since the tiling is much more modified.
However we have to find other faces than those involved in the coincidence, and this makes this case a bit more technical.
We need several lemmas.\\

The first lemma show how to ``fold'' inside the window a path whose endpoints are in the window (Fig.~\ref{fig:path_folding}):

\begin{lemma}\label{lem:fold}
  Let $E$ be a $d$-plane of $\mathbb{R}^n$.
  If two points of $\mathbb{Z}^n$ project inside $W_E$, then they are connected by a path of unit edges of $\mathbb{Z}^n$ which projects into $W_E$.
\end{lemma}

\begin{proof}
  We shall swap the edges of the path so that it remains inside the window.
  We proceed by induction on the length $k$ of the path.
  This is trivial for $k=0$.
  Consider a path of length $k>0$ and assume that it wanders outside the window (otherwise there is nothing to prove).
  Consider the first edge, say $\vec{a}$ which cross the window's boundary, say on face $f$.
  To get back in the window, the path must contain a further edge, say $\vec{b}$, which goes back to the other side of the hyperplane which contains $f$.  
  Using $\vec{b}$ instead of $\vec{a}$ leads to a point inside the window.
  By swapping the edges $\vec{a}$ and $\vec{b}$ in the path, the endpoints are unmodified, but the path stays inside the window after using edge $\vec{b}$.
  The remaining part (after edge $\vec{b}$) has length less than $k$ and can, by induction, be folded inside the window.
  The whole path now stays inside the window.
\end{proof}

\begin{figure}[hbtp]
  \centering
  \includegraphics[width=\textwidth]{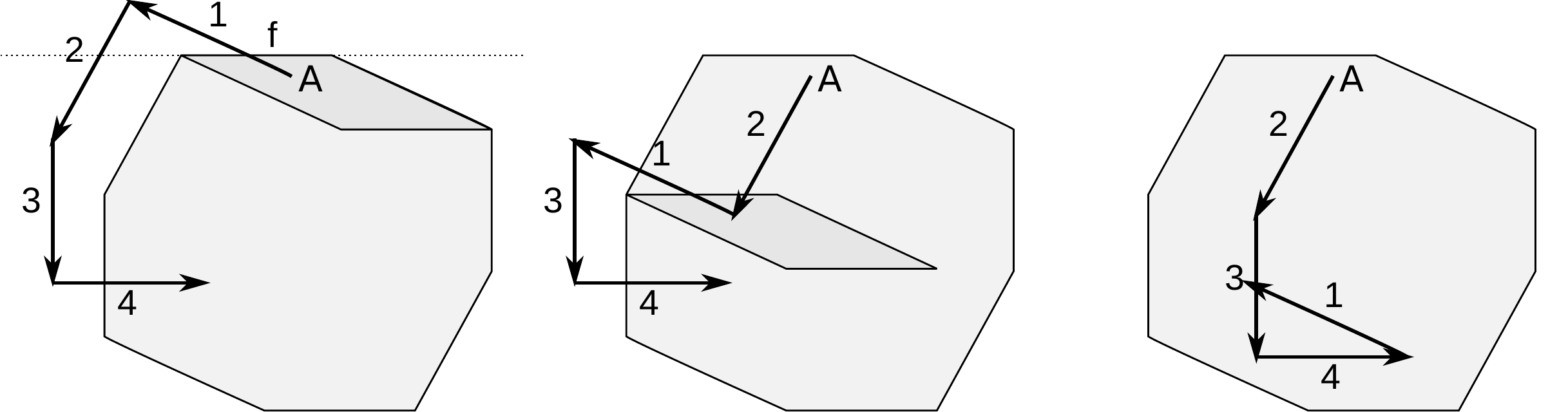}
  \caption{How to permute the edges of a path with endpoints in the window so that it lies completly in the window.}
  \label{fig:path_folding}
\end{figure}

The second lemma relies on the first one to show that $r$-patterns can force an integer point which projects inside the window to still project inside the window:

\begin{lemma}\label{lem:point_in_window}
  Let $E$ and $F$ in $G'(n,d)$, $E\neq F$.
  If $\vec{x}\in\mathbb{Z}^n\cap[-r,r]^n$ projects into the window of $E$ but not into the window of $F$, then any planar tiling with a slope parallel to $F$ contains an $r$-pattern which does not appear in some planar tiling with a slope parallel to $E$.
\end{lemma}

\begin{proof}
  Consider a path made of unit edge of $\mathbb{Z}^n$ from the origin $O$ (which projects in $W_E$) to $\vec{x}$.
  We can assume that each coordinate of the vertices along this path varies in a monotonic way (otherwise it suffices to permute edges and cancel consecutive opposite edges).
  Thus, the vertices on this path are in $[-r,r]^n$.
  We fold this path so that it projects into $W_E$ (Lem.~\ref{lem:fold}).
  Since folding amounts to permute edges, the vertices on this path are still in $[-r,r]^n$ (Fig.~\ref{fig:forcing_point_in_window}, top-left).\\

  In $F'$ (the complementary space of $F$ which contains $W_F$), the projection of this path still connects $O$ to $\vec{x}$, but $\vec{x}$ is now outside $W_F$.
  Let $\vec{y}\in\mathbb{Z}^n$ denote the first vertex of this path which does not project in $W_F$ and $z$ be a vertex of the face of $W_F$ crossed by the edge which arrives in $\vec{y}$ (Fig.~\ref{fig:forcing_point_in_window}, top-right).\\

  Consider the two following translated windows: the first one by the vector $\vec{z}-\vec{z'}$ which maps the face of $W_F$ containing $z$ to the parallel face of $W_F$ ($W_1$ on Fig.~~\ref{fig:forcing_point_in_window}, left-right), and the second one by the vector $\vec{y}-\vec{z}$ ($W_2$ on Fig.~~\ref{fig:forcing_point_in_window}, bottom-right).
  The intersection of both theses windows in $F'$ is non-empty.
  However, the intersection of the two same windows in $E'$ (Fig.~~\ref{fig:forcing_point_in_window}, bottom-left) has empty interior (they are separated by the face which contains $\vec{z}$; the intersection is even empty except if $\vec{y}$ belongs to the boundary of $W_F$).\\

  The vector which maps $W_2$ onto $W_1$ is thus $\vec{z}-\vec{z'}+\vec{y}-\vec{z}=\vec{y}-\vec{z'}$: since $\vec{y}\in [-r,r]^n$ and $\vec{z'}\in[0,1]^n$, it is in $[0,1]^{n+1}$.
  This still holds if we translate both windows by a vector in $\pi'\mathbb{Z}^n$ such that the intersection of $W_1$ and $W_2$ in $F'$ intersect $W_F$.
  By density of $\pi'\mathbb{Z}^n$ in $F'$, this ensures that any planar tiling with a slope parallel to $F$ contains an $r$-pattern.
  Since we have seen that the same region has empty interior in $E'$, this pattern does not appear in some planar tiling with a slope parallel to $E$.
\end{proof}

\begin{figure}
  \includegraphics[width=\textwidth]{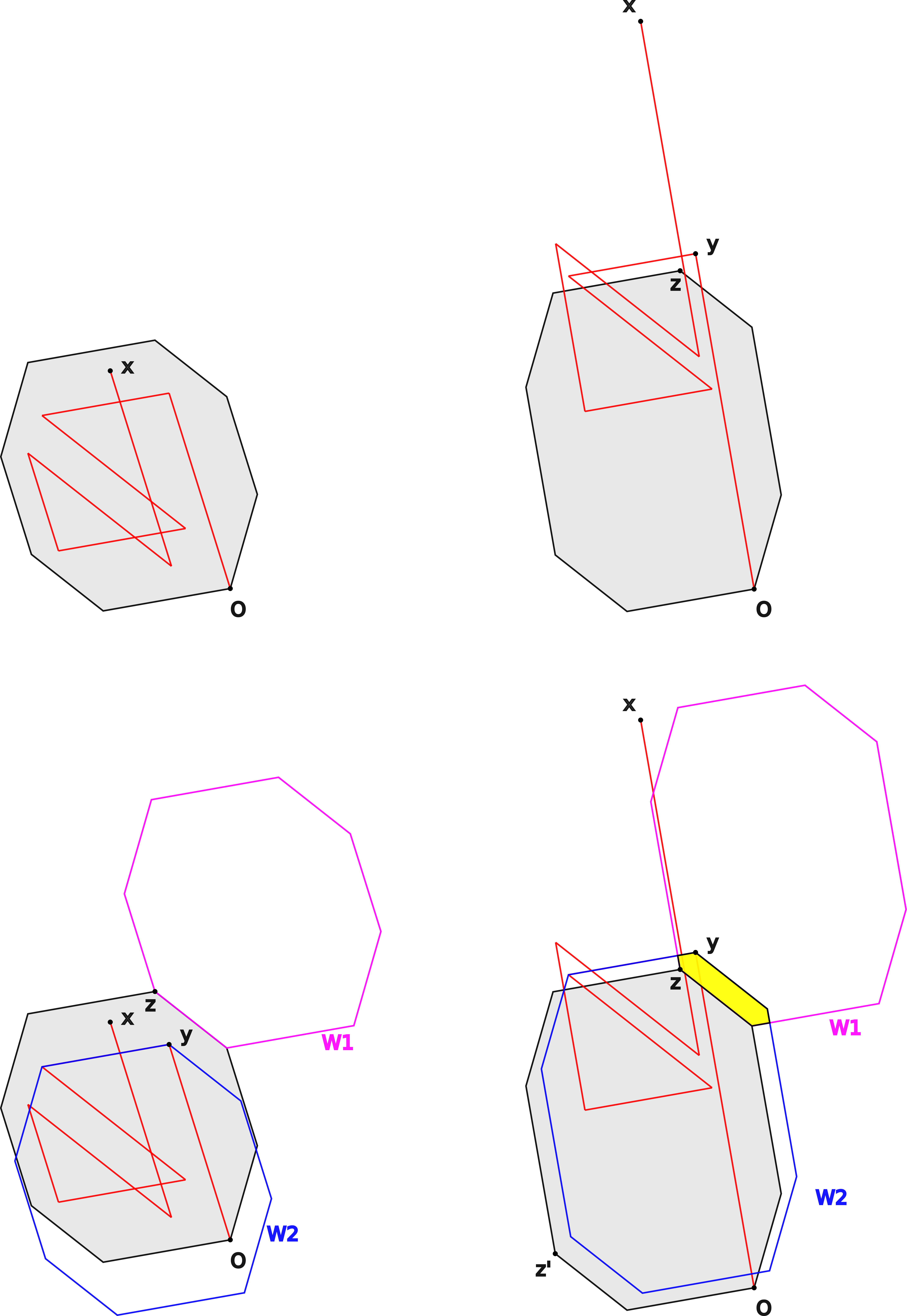}
  \caption{Lemma~\ref{lem:point_in_window}: $r$-patterns can force a point $x$ to project inside the window.}
  \label{fig:forcing_point_in_window}
\end{figure}

The last lemma relies on the second one to show that $r$-patterns can force two unit faces whose projections intersect in the window to still intersect in the window:

\begin{lemma}\label{lem:big_break}
  Let $E$ and $F$ in $G'(n,d)$, $E\neq F$.
  If two faces of an $r$-coincidence of $E$ do not intersect in $W_F$, then any planar tiling of slope $F$ has an $r$-pattern which does not appear in some planar tiling of slope $E$.
\end{lemma}

\begin{proof}
  Consider two faces $f$ and $f'$ which intersect in $W_E$ but not in $W_F$.
  There is an edge of $f$ which crosses $f'$ in $W_E$ but not $W_F$.
  Let $\vec{x}$ and $\vec{y}$ denote the endpoints of this edge.
  Consider the parallelotope defined by the face $f'$ and the vector $\vec{x}-\vec{y}$.
  In $W_E$, it contains either $\vec{x}$ or $\vec{y}$ - say $\vec{x}$.
  In $W_F$, it contains neither $\vec{x}$ nor $\vec{y}$.
  The point $\vec{x}$ moved outside this parallelotope when changing $E$ to $F$.
  Since this parallelotope is the intersection of translated window (one for each face), there is at least one of these translated window which contains $\vec{x}$ in $E$ but not in $F$.
  Since both $\vec{x}$ and the vertices of the parallelotope have coordinates in $[-r,r]$, Lemma~\ref{lem:point_in_window} ensures that any planar tiling of slope $F$ has an $r$-pattern which does not appear in some planar tiling of slope $E$.
\end{proof}

Lemmas~\ref{lem:small_break} and \ref{lem:big_break} show that the $r$-patterns prevent both small and big breaks of $r$-coincidences.
The second part of Theorem~\ref{th:main} follows:

\begin{proposition}\label{prop:coincidences2patterns}
  If a slope $E$ in $G'(n,d)$ is characterized by $r$-coincidences, then it is characterized by all the $r$-patterns which do not appear in the planar tiling with slope $E$.
\end{proposition}

\section{A typical example}
\label{sec:typical}

Let us here consider a "typical" generic algebraic $2$-plane of $\mathbb{R}^4$.
We first choose an irreducible polynomial of small degree, say $X^3 + X^2 - X + 1$.
Let $a\simeq -1.834$ be a real root of this polynomial (actually, the unique real root).
We then choose two vectors in $\mathbb{Q}[a]$, say
$$
(a, 2a^2 + 2a + 1, 2a + 2, 2a^2)
\quad\textrm{and}\quad
(2a^2 + 2a, a^2 + a + 1, 2, 2a^2 + 2a + 1).
$$
We check that the plane defined by these two vectors is generic.
Let $E$ denote this plane.
Its Grassmann coordinates are
\begin{eqnarray*}
G_{12}&=&-6a^2 + 3,\\
G_{13}&=&-4a^2 - 6a + 4,\\
G_{14}&=&-4a^2 + 7a - 2,\\
G_{23}&=&2a^2 - 2a + 2,\\
G_{24}&=&4a^2 + 6a - 3,\\
G_{34}&=&10a - 2
\end{eqnarray*}
Fig.~\ref{fig:typical_tiling} depicts a corresponding canonical planar tiling.\\

\begin{figure}[hbtp]
\includegraphics[width=\textwidth]{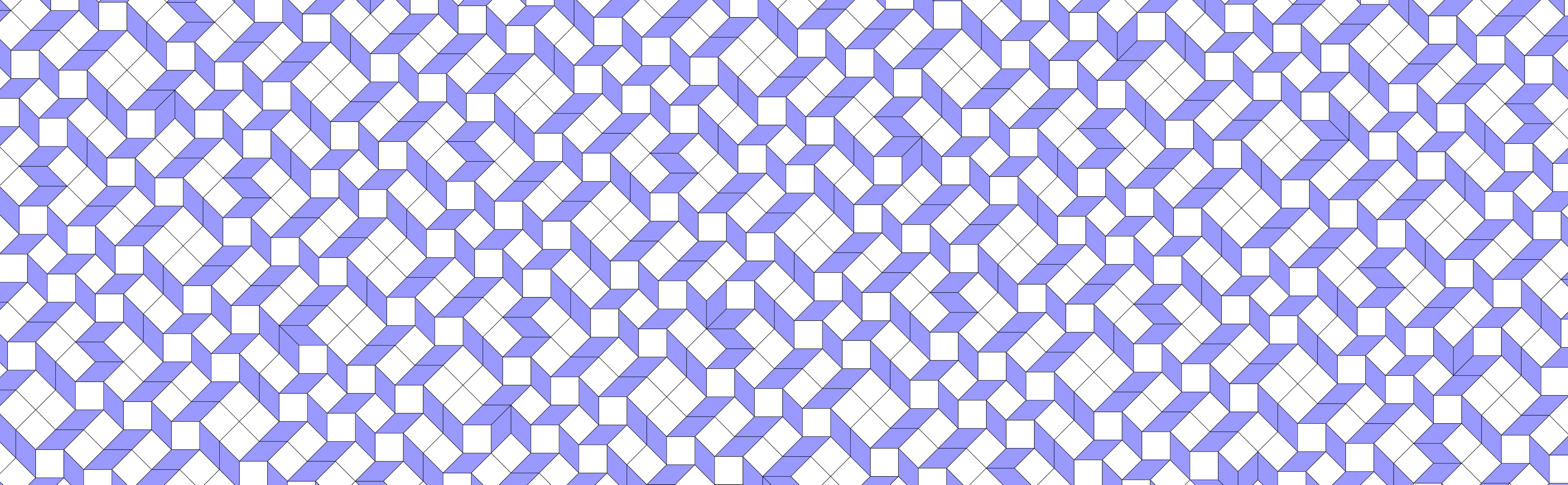}
\caption{A $4\to 2$ canonical planar tiling with a typical generic algebraic slope.}
\label{fig:typical_tiling}
\end{figure}

In the case of a $2$-plane in $\mathbb{R}^4$, coincidences are triple of unit non-parallel segments of $\mathbb{Z}^4$ which are concurrent once projected into the window.
To find them, we thus have to find three points of $\mathbb{R}^4$ which project onto the same point in the window, each of them with three integer entries and one (possibly) real one, with no two points with their real entry at the same position.
Depending on where the real entries are, this yields four {\em coincidence types}\footnote{And no more than $n^{(d+1)(n-d+1)}$ for a $d$-plane of $\mathbb{R}^n$.}.\\

\noindent Let us, for example, consider a coincidence of the type
$$
(a_1,a_2,a_3,r_1),
\qquad
(a_4,a_5,r_2,a_6),
\qquad
(a_7,r_3,a_8,a_9),
$$
where the $a_i$'s denote the integer entries and the $r_j$'s the real ones.
The fact that these four points project in the window on a single point yields a system of equations that can be written ({\em cf} proof of Proposition~\ref{prop:eq_coincidence2})
$$
\left(\begin{array}{r|rrr|r|r}
a_1 -  a_2 & 0 & 0 & 0 & &\\
a_2 -  a_3 & 0 & 0 & 0 & U & 0\\
a_3 & 0 & -1 & 0 & &\\
- a_4 & 1 & 0 & 0 & &\\
\hline
a_1 -  a_{6} & 0 & 0 & 0 & &\\
a_2 & 0 & 0 & -1 & 0 & U\\
a_3 -  a_8 & 0 & 0 & 0 & &\\
- a_9 & 1 & 0 & 0 & &
\end{array}\right)
\left(\begin{array}{c}
  1\\r_1\\r_2\\r_3\\\lambda_{11}\\\lambda_{12}\\\lambda_{21}\\\lambda_{22}
\end{array}\right)
=0,
$$
where the $\lambda_{ij}$'s are real numbers and $U$ is the the matrix whose columns are generators of $E$.
Denote by $M$ the above matrix.
In order to have a coincidence, the determinant of $M$ must be zero.
This determinant is a polynomial in $a$ whose coefficients are integer linear combinations of the $a_i$'s.
The coefficients of $a^i$ must thus be zero for any $i$ less than the algebraic degree of $a$.
Here, this yields (for $i=0,1,2$):
\begin{eqnarray*}
0&=&17 a_1 + 6 a_2 - 30 a_3 + 10 a_4 - 6 a_5 - 27 a_7 + 30 a_8,\\
0&=&56 a_1 - 4 a_2 - 69 a_3 - 26 a_4 + 4 a_5 + 30 a_6 - 30 a_7 + 69 a_8-30 a_9,\\
0&=&32 a_1 - 7 a_2 - 45 a_3 + 4 a_4 + 7 a_5 - 12 a_6 - 36 a_7 + 45 a_8 + 12 a_9.
\end{eqnarray*}
Finding the $a_i$'s thus amounts to compute the kernel of a $3\times 9$ matrix\footnote{And $k\times (d+1)(n-d+1)$ for a $d$-plane of $\mathbb{R}^n$ with entries in a number field of degree $k$.}.
Since coincidences form a $\mathbb{Z}$-module (seen as tuples of $n-d+1$ points in $\mathbb{R}^n$ with pointwise operations), it suffices to consider the coincidences associated with a basis of the above kernel.
Here, the kernel has dimension $6$ and contains, for example, the vector
$$
(a_1,\ldots,a_9)=(3, -3, 3, 3, 3, 2, -5, -3, -3).
$$
To get the equation associated with such a basis vector, it suffices to replace in $M$ the $a_i$'s by their values and to express $U$ with the Grassmann coordinates of $E$: the nullity of $\det(M)$ yields the equation.
The above vector, for example, yields the equation
$$
0=5 G_{12} G_{13} - 6 G_{12} G_{14} - 6 G_{13} G_{14} + 8 G_{12} G_{34}.
$$
We can also find in the kernel of $M$ the vector $(1,r_1,r_2,r_3,\lambda_{11},\lambda_{12},\lambda_{21},\lambda_{22})$ to compute explicitly the $r_i$'s:
\begin{eqnarray*}
r_1&=&-4a^2 - 2a + 2\simeq -7.853,\\
r_2&=&-\tfrac{64}{17}a^2 - \tfrac{52}{17}a + \tfrac{163}{17} \simeq 2.478,\\
r_3&=&-a^2 - 5a \simeq -0.186.
\end{eqnarray*}
The coincidence is thus completly determined (up to translation):
$$
(3,~-3,~3,~-7.853),
\qquad
(3,~3,~2.478,~2),
\qquad
(-5,~-0.186,~-3,~-3).
$$
Its entries are less than $r=8$ in modulus.
We improve to $r=5$ by translating it by $(0,0,0,3)$.
Section~\ref{sec:coincidences2patterns} ensures that this coincidence is enforced by $5$-patterns.\\

Many of the coincidences actually yield the trivial equation $0=0$: these are the ``degenerated'' coincidences where the $n-d+1$ faces are already concurrent in $\mathbb{R}^n$ (thus in a point of $\mathbb{Z}^n$).\\

Proceeding similarly for each of the four possible types of coincidence yields all the coincidences and the corresponding equations.
We get the equations:
\begin{eqnarray*}
0&=& 5 G_{12} G_{13} - 6 G_{12} G_{14} - 6 G_{13} G_{14} + 8 G_{12} G_{34},\\
0&=& 17 G_{12} G_{13} + 28 G_{12} G_{14} - 36 G_{13} G_{14} + 51 G_{13} G_{24} - 42 G_{12} G_{34},\\
0&=& 7 G_{12} G_{23} - 6 G_{14} G_{23} + 4 G_{12} G_{24} - 9 G_{23} G_{24} - 2 G_{12} G_{34},\\
0&=& 14 G_{12} G_{23} - 9 G_{14} G_{23} - 2 G_{12} G_{24} + 24 G_{23} G_{24},\\
0&=& 49 G_{13} G_{23} + 2 G_{14} G_{23} + 12 G_{13} G_{24} - 10 G_{13} G_{34} + 2 G_{23} G_{34},\\
0&=& 38 G_{13} G_{23} - 72 G_{14} G_{23} - 8 G_{13} G_{24} +  G_{13} G_{34} + 103 G_{23} G_{34},\\
0&=& 13 G_{14} G_{23} - 6 G_{13} G_{24} + 2 G_{14} G_{24} + 8 G_{14} G_{34} - 7 G_{24} G_{34},\\
0&=& 38 G_{14} G_{23} + 17 G_{13} G_{24} - 166 G_{14} G_{24} + 26 G_{14} G_{34} + 74 G_{24} G_{34}.
\end{eqnarray*}
We also have to add to these equations the Plücker relations.
There is only one relation for $2$-planes in $\mathbb{R}^4$:
$$
0=G_{12}G_{34}-G_{13}G_{24}+G_{14}G_{23}.
$$
Last, we have to normalize Grassmann coordinates, {\em e.g.}, by setting $G_{12}=1$.
One checks (using, {\em e.g.}, \cite{sagemath}) that the obtained system of polynomial equations is $0$-dimensional.
In other words, $E$ is characterized by coincidences, and Prop.~\ref{prop:coincidences2patterns} ensures that the canonical planar tilings with slope $E$ are characterized by patterns.
More precisely, computing the first four coincidences show that the corresponding equations suffice to characterize $E$.
By suitably translating these coincidences, there entries have modulus at most $r=29$, so that $29$-patterns characterize the canonical planar tilings with slope $E$.\\

Remark that the equations and coincidences we find this way depend on the basis of the kernel of $M$ we use.
In particular, it seems worth finding a basis with short vectors to get coincidences with small entries, hence characterize tilings with small patterns.
The above equations and coincidences have been obtained from the Block-Korkine-Zolotarev reduction of the kernel basis of $M$.\\

Let us mention that the main result of \cite{BF17} ensures that the canonical planar tilings with slope $E$ are not characterized by patterns {\em among all the canonical tilings}, that is, there are {\em non-planar} canonical tilings with the same $29$-patterns.
Theorem~\ref{th:main} just ensures characterization by patterns among {\em planar} canonical tiling.

\section{Ammann-Beenker tilings}
\label{sec:ammann_beenker}

The Ammann-Beenker tilings first appeared in \cite{GS86} (Chap. 10).
Defined as the fixed-points of some substitution \cite{AGS92}, they have also been shown in \cite{Bee82} to be the canonical planar tilings whose slope $E$ is the $2$-plane of $\mathbb{R}^4$ generated by
$$
\left(\cos\tfrac{k\pi}{4}\right)_{0\leq k<4}
\qquad\textrm{and}\qquad
\left(\sin\tfrac{k\pi}{4}\right)_{0\leq k<4}.
$$
Alternative generators are
$$
(-1,0,1,a)
\qquad\textrm{and}\qquad
(0,1,a,1),
$$
where $a=\sqrt{2}$ is the positive root of $X^2-2$.\\

\begin{figure}[hbtp]
\includegraphics[width=\textwidth]{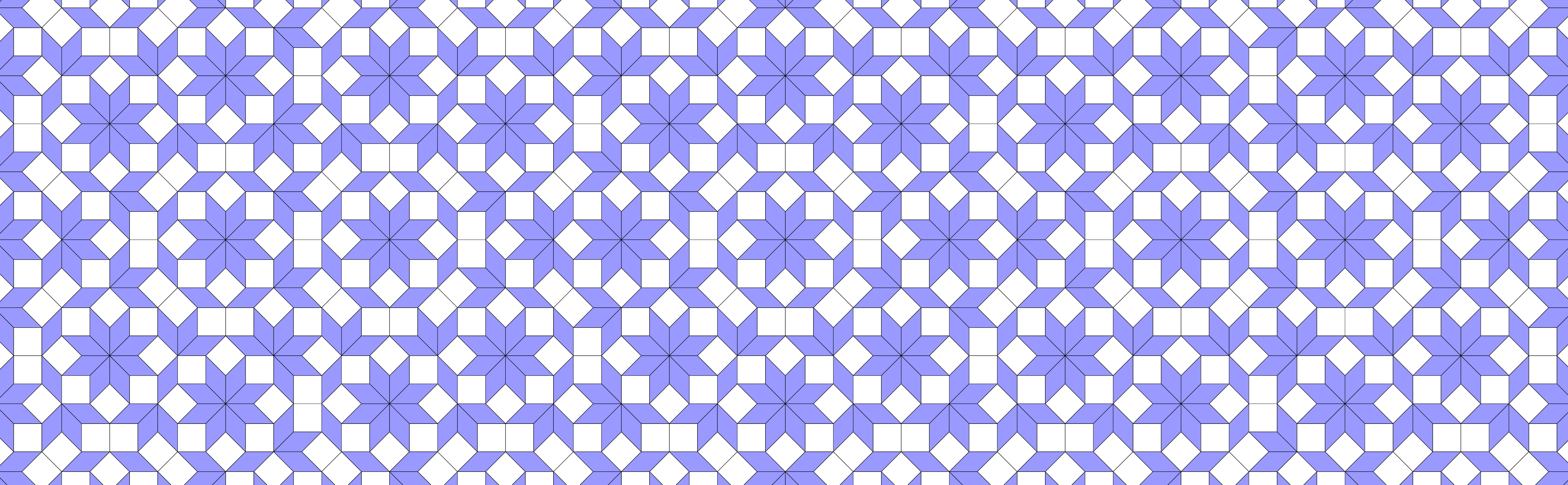}
\caption{An Ammann-Beenker tiling.}
\label{fig:beenker_tiling1}
\end{figure}

The smallest rational space containing $E$ is $\mathbb{R}^4$: it is thus generic.
The Grassmann coordinates of $E$ are
$$
G_{12}=G_{14}=G_{23}=G_{34}=1\\
\qquad\textrm{and}\qquad
G_{13}=G_{24}=a.
$$
An exhaustive search for coincidences (as in the previous example) yields the system:
$$
G_{12}=G_{14}=G_{23}=G_{34}.
\qquad\textrm{and}\qquad
G_{13}G_{24}=G_{12}^2.
$$
With $G_{12}=1$ this becomes
$$
G_{12}=G_{14}=G_{23}=G_{34}=1\\
\qquad\textrm{and}\qquad
G_{13}G_{24}=2.
$$
The Plücker relation also becomes $G_{13}G_{24}=2$.
There is thus a continuum of slopes with the same coincidences as the Ammann-Beenker tilings.
Prop.~\ref{prop:patterns2coincidences} ensures that Ammann-Beenker tilings are not characterized by patterns.
More precisely, for any $r$, we get a tiling with the same $r$-patterns by choosing $G_{13}$ close enough to $a$.
Fig.~\ref{fig:beenker_tiling2} illustrates the case $G_{13}=\tfrac{3}{2}$.\\

\begin{figure}[hbtp]
\includegraphics[width=\textwidth]{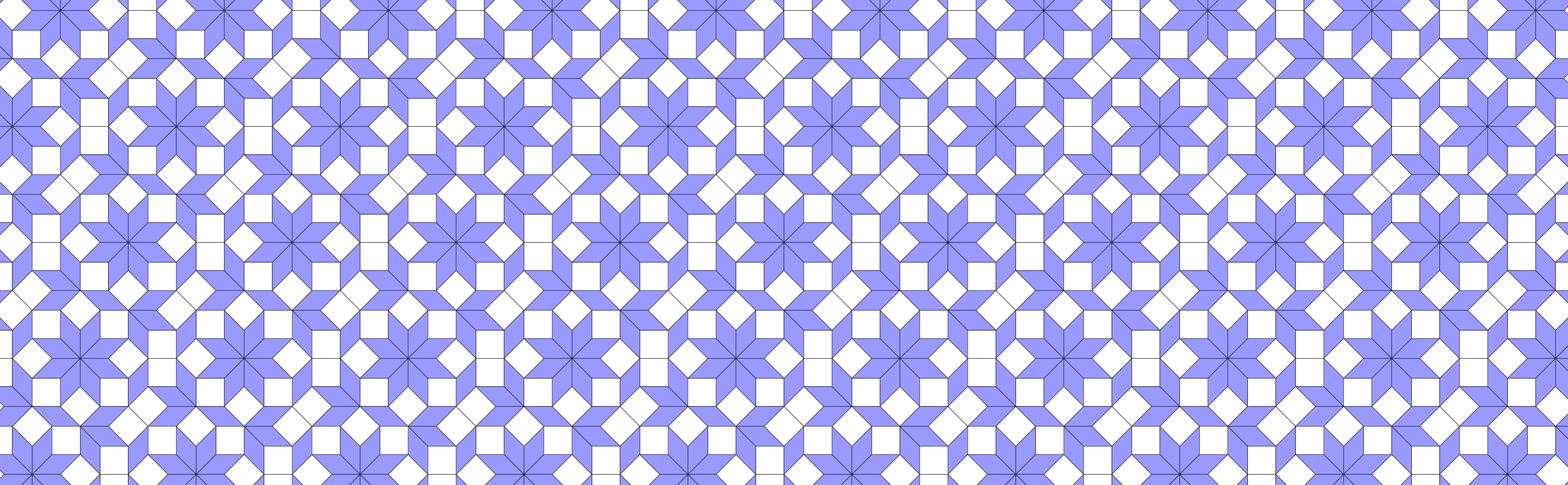}
\caption{A tiling with the same vertex-atlas as Ammann-Beenker tilings but with a different slope.}
\label{fig:beenker_tiling2}
\end{figure}

This result first appeared in \cite{Bur88} (relying on the particular value $a=\sqrt{2}$, although the above holds for any positive irrational value), see also \cite{BF13,Kat95}.
This is a particular case of the so-called {\em $4n$-fold} canonical planar tilings (the canonical planar tilings which have the same finite patterns as their image by a rotation of angle $\frac{2\pi}{4n}$ for some $n\geq 2$), see \cite{BF15a}.

\section{Penrose tilings}
\label{sec:penrose}

The Penrose tilings first appear in \cite{Pen78}.
Defined as the fixed-points of some substitution, they have also been shown in \cite{DB81} to be the canonical planar tilings whose slope $E$ is the $2$-plane of $\mathbb{R}^5$ generated by
$$
\left(\cos\tfrac{2k\pi}{5}\right)_{0\leq k<5}
\qquad\textrm{and}\qquad
\left(\sin\tfrac{2k\pi}{5}\right)_{0\leq k<5}.
$$
Alternative generators are
$$
(a,0,-a,-1,1)
\qquad\textrm{and}\qquad
(-1,1,a,0,-a)
$$
where $a=\tfrac{1+\sqrt{5}}{2}$ is the positive root of $X^2-X-1$, that is, the golden ratio.
The Grassmann coordinates of $E$ are (recall that $G_{ij}=-G_{ij}$ by convention):
$$
G_{12}=G_{23}=G_{34}=G_{45}=G_{51}=a,
$$
$$
G_{13}=G_{35}=G_{52}=G_{24}=G_{41}=1.
$$

\begin{figure}[hbtp]
\includegraphics[width=\textwidth]{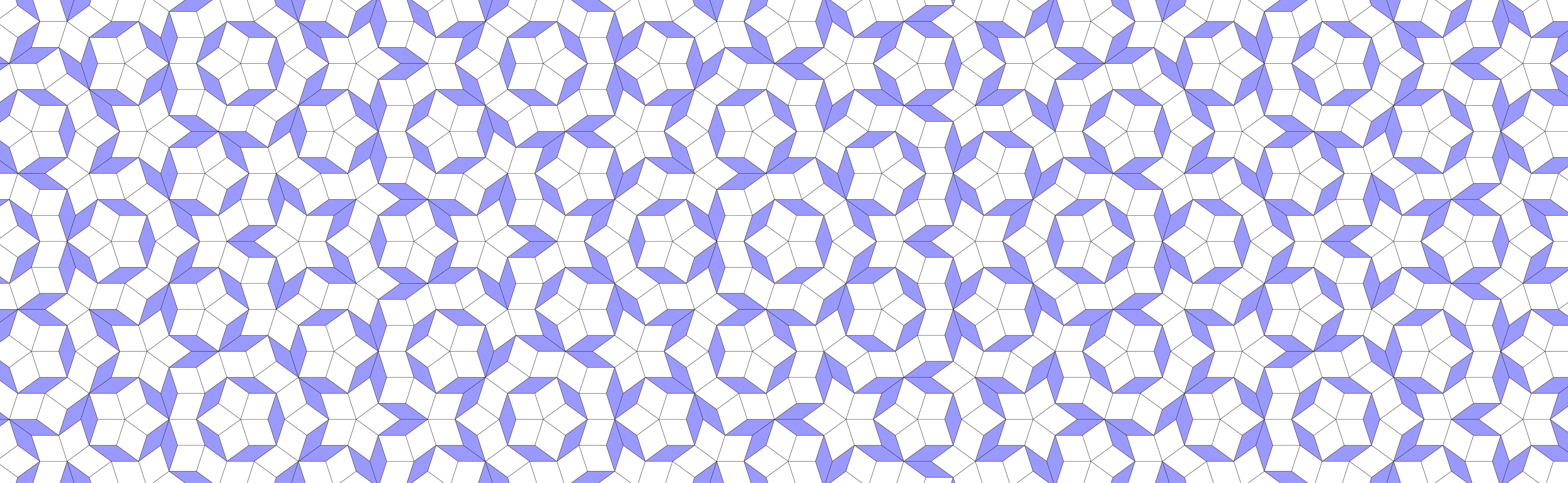}
\caption{A Penrose tiling.}
\label{fig:penrose_tiling}
\end{figure}

The point of this example is that the slope $E$ of Penrose tilings is not generic.
Indeed, the generators are orthogonal to $(1,1,1,1,1)$, so that $E$ is included in a $4$-dimensional rational space.
One checks that it is the smallest one.
It is actually possible to directly define Penrose tilings as {\em non-canonical} projection tilings in a $4$-dimensional space (see, {\em e.g.}, \cite{BG13}, Sec.~7.3).
However, we are here not interested in the Penrose tilings themselves, but we want to characterize them among all the canonical planar tilings with different slopes.
Since almost all these tilings are generic, they cannot be defined as projection tiling in a $4$-dimensional space.
This is why the usual trick to reduce to the generic case seems to be useless in our case.
Let us see more in details which problems occur.\\

We can still compute coincidences and the corresponding equations, as usually.
Among other ones, we get the equations
$$
G_{12}=G_{23}=G_{34}=G_{45}=G_{51},
$$
$$
G_{13}=G_{35}=G_{52}=G_{24}=G_{41}.
$$
With the normalization $G_{13}-1$, the Plücker relation
$$
0=G_{12}G_{34}-G_{13}G_{24}+G_{14}G_{23}
$$
becomes
$$
0=G_{12}^2-1-G_{12}.
$$
This enforces $G_{12}=a$ (the algebraic conjugate is impossible because it is negative, hence corresponds to a non-achievable slope).
The slope $E$ of Penrose tilings is thus characterized by coincidences.\\

Lemmas~\ref{lem:small_break} and \ref{lem:big_break} then ensure that there are forbidden patterns such that, for any slope $F\neq E$, the coincidences "break", that is, at least one of the forbidden patterns has a non-empty interior region in the window $W_F$.
In particular, this prevents $F$ to be generic, because otherwise $\pi'\mathbb{Z}^5$ would be dense in $W_F$, hence the non-empty interior region would contain a projected integer point, that is, a forbidden pattern would appear in the canonical planar tilings of slope $F$.
However, if $F$ is not generic, then one could imagine that although the region of the forbidden pattern has non-empty interior, no point of $\pi'\mathbb{Z}^5$ falls into it (because this set is not dense).   
For example, consider the slope $F$ generated by
$$
(\tfrac{8}{5},0,-\tfrac{8}{5},-1,1)
\qquad\textrm{and}\qquad
(-1,1,\tfrac{5}{3},0,-\tfrac{5}{3}).
$$
It is a small modification of $E$: the golden ratio is replaced in each generator by a different continued fraction approximation.
Fig.~\ref{fig:penrose_coincidence} compares the region of one of the Penrose coincidence both in the window of $E$ and $F$.
Since $F$ is non-generic, integer points are not dense in $W_F$, and it is possible that none falls into the region of the broken coincidence, even it this latter has non-empty interior (but maybe it is not possible to avoid all the coincidences and their translations?).\\

\begin{figure}[hbtp]
\includegraphics[width=0.48\textwidth]{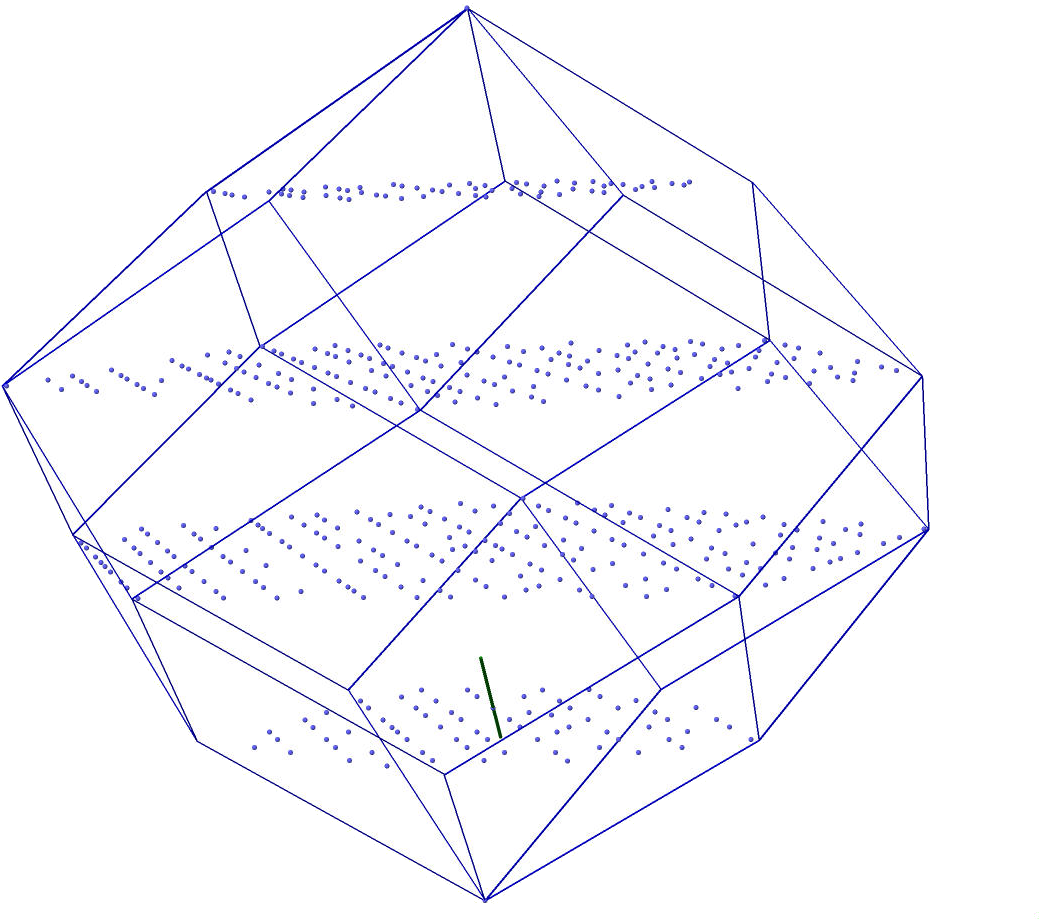}
\hfill
\includegraphics[width=0.48\textwidth]{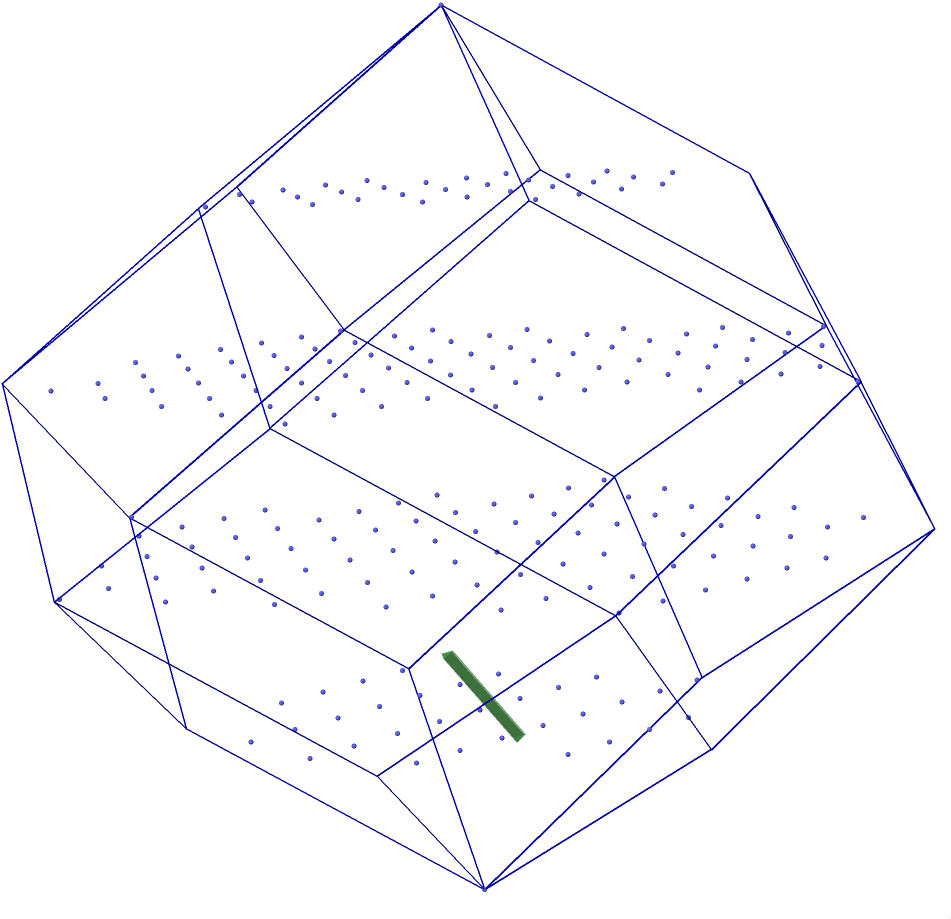}
\caption{
Left: the window $W_E$ is a rhombic icosahedron and the points of $\pi'\mathbb{Z}^5$ densely fill four parallel planes in $W_E$ (some are depicted); the region of the coincidence is a segment which goes through one of these planes but does not contain a point of $\pi'\mathbb{Z}^5$.
Right: when $E$ is changed to $F$, the coincidence breaks in an non-empty interior region (namely a triangular prism).
But $F$ is even more non-generic than $E$ ($\pi'\mathbb{Z}^5$ is discrete) and no point of $\pi'\mathbb{Z}^5$ fall into the region: we could get a planar tiling of slope $F\neq E$ without forbidden pattern (if no other regions of coincidences contain a point of $\pi'\mathbb{Z}^5$).
}
\label{fig:penrose_coincidence}
\end{figure}

Nevertheless, the Penrose tilings have been proven (by other arguments) to be characterized by patterns.
Namely, any Penrose tiling contains $7$ different $0$-patterns (up to rotation, Fig.~\ref{fig:penrose_patterns}), and any tiling with the same tiles\footnote{Not even assumed to be planar.} without other $0$-pattern is necessarily a Penrose tiling (see \cite{Sen95}, Th.~6.1).\\

\begin{figure}[hbtp]
\includegraphics[width=\textwidth]{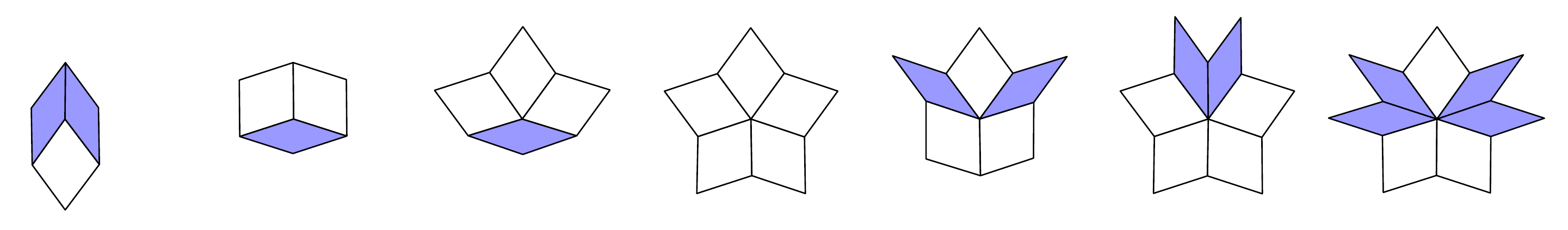}
\caption{The $0$-patterns (up to isometry) of the Penrose tilings.}
\label{fig:penrose_patterns}
\end{figure}

It is not very hard to find a set of forbidden patterns such that whenever none of them appear in a canonical tiling, the $0$-patterns of this tiling are $1$-patterns of Penrose tilings.
Such a set of forbidden patterns is not unique.
Fig.~\ref{fig:penrose_forbidden_patterns} provides a rather light one.\\

\begin{figure}[hbtp]
\includegraphics[width=\textwidth]{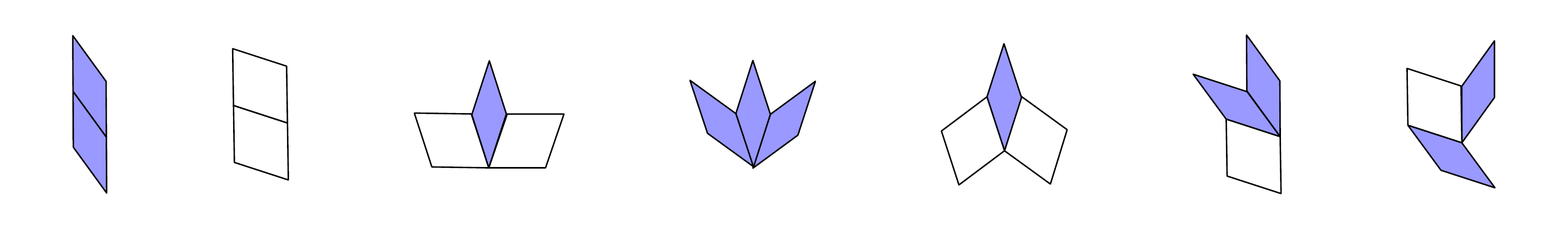}
\caption{Forbidden patterns (up to isometry) characterizing Penrose tilings.}
\label{fig:penrose_forbidden_patterns}
\end{figure}

Since these forbidden patterns characterize the slope $E$, for any slope $F\neq E$, generic or not, at least one point of $\pi'\mathbb{Z}^5$ falls into the region of a forbidden pattern.
These regions may thus be good candidates to extend the notion of coincidence to the non-generic case.
Computations show that the first forbidden pattern (from left to right on Fig.~\ref{fig:penrose_forbidden_patterns}) has an empty region (this pattern thus does not appear after a sufficiently small modification of $E$), the second one has a region with an empty interior which intersects a plane in the closure of $\pi'\mathbb{Z}^5$, and the other ones have non-empty interior but intersect the planes in the closure of $\pi'\mathbb{Z}^5$ only on their boundaries.
Fig.~\ref{fig:penrose_forbidden_regions} illustrates this.\\

\begin{figure}[hbtp]
\includegraphics[width=0.48\textwidth]{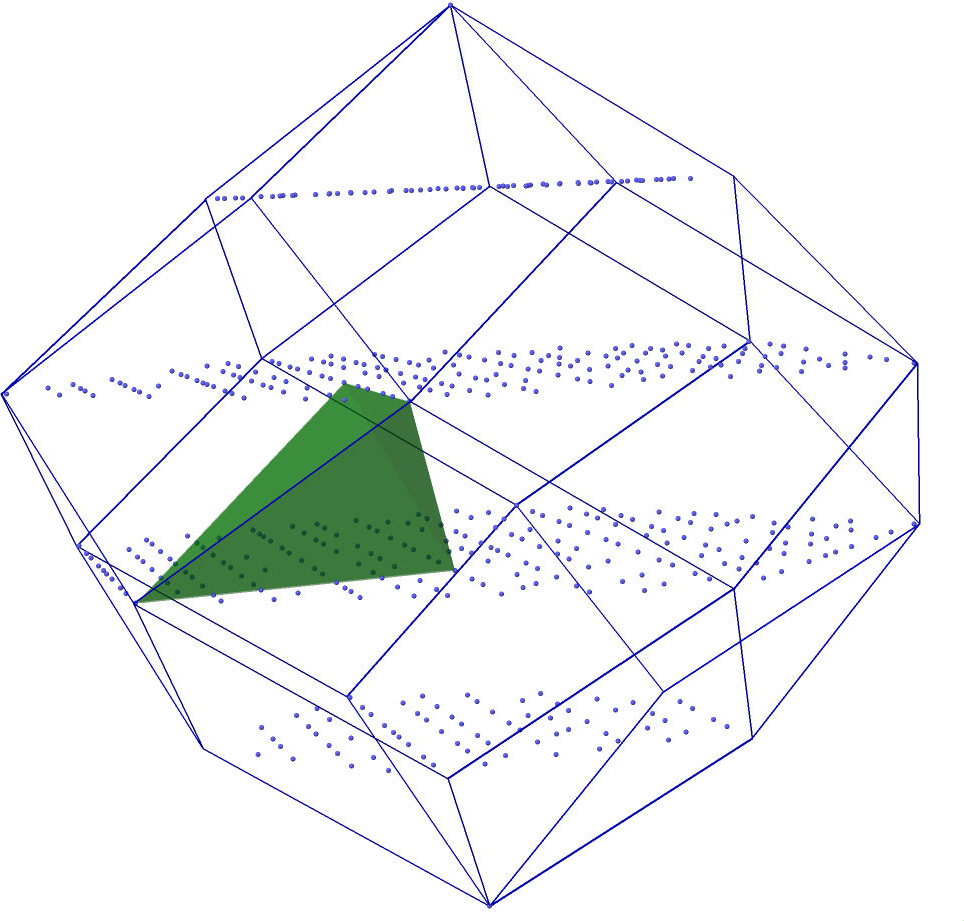}
\hfill
\includegraphics[width=0.48\textwidth]{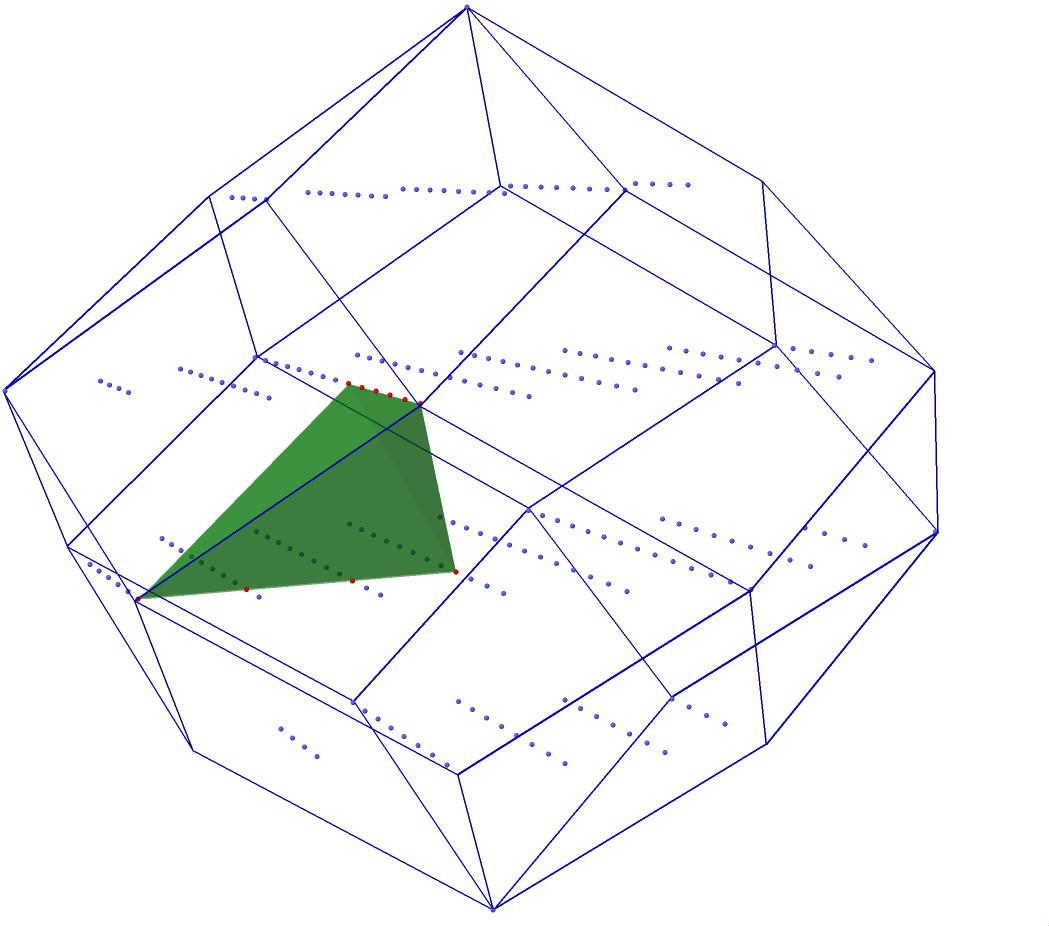}
\caption{
Left: the window $W_E$, the closure of the projection of $\mathbb{Z}^5$ and the region of one of the forbidden patterns which characterize $E$ (namely the fifth one on Fig.~\ref{fig:penrose_forbidden_patterns}).
The region has non-empty interior but its intersection with the planes densely filled by $\pi'\mathbb{Z}^5$ have empty interior (this is the union of two segments): the forbidden pattern does not appear.
Right: the same picture for the rational plane $F$ already used for Fig.~\ref{fig:penrose_forbidden_regions}.
The region has non-empty interior and contains (on its boundary) infinitely many points (because $\pi'\mathbb{Z}^5$ is a lattice in $W_F$): the forbidden pattern appears.
}
\label{fig:penrose_forbidden_regions}
\end{figure}

Actually, Penrose tilings correspond to specific affine slopes parallel to $E$ (namely those which contains a point of $\mathbb{Z}^5$).
The tilings defined by all the slopes parallel to $E$ are known as {\em generalized Penrose tilings} \cite{KP87}.
They may have different patterns, but the set of all generalized Penrose tilings is nevertheless defined by forbidden patterns.
An exemple of set of forbidden patterns is given in \cite{BF15b}, and one checks that the associated regions in the window are either empty or polygons (hence have empty interior).
This is not surprising, since when the slope ranges through all the planes parallel to $E$, the points of $\mathbb{Z}^5$ densely project into the window (and not in parallel planes as for Penrose tilings).

\bibliographystyle{alpha}
\bibliography{coincidences}

\begin{thebibliography}{HKSW16}

\bibitem[AGS92]{AGS92}
R.~Ammann, B.~Gr{\"u}nbaum, and G.~C. Shephard.
\newblock Aperiodic tiles.
\newblock {\em Discrete {\&} Computational Geometry}, 8:1--25, 1992.

\bibitem[Bee82]{Bee82}
F.~P.~M. Beenker.
\newblock Algebric theory of non periodic tilings of the plane by two simple
  building blocks: a square and a rhombus.
\newblock Technical Report TH Report 82-WSK-04, Technische Hogeschool
  Eindhoven, 1982.

\bibitem[BF13]{BF13}
N.~B{\'e}daride and Th. Fernique.
\newblock {\em Aperiodic Crystals}, chapter The {A}mmann--{B}eenker tilings
  revisited, pages 59--65.
\newblock Springer Netherlands, Dordrecht, 2013.

\bibitem[BF15a]{BF15a}
N.~B{\'e}daride and Th. Fernique.
\newblock No weak local rules for the 4p-fold tilings.
\newblock {\em Discrete \& Computational Geometry}, 54:980--992, 2015.

\bibitem[BF15b]{BF15b}
N.~B{\'e}daride and Th. Fernique.
\newblock When periodicities enforce aperiodicity.
\newblock {\em Communications in Mathematical Physics}, 335:1099--1120, 2015.

\bibitem[BF17]{BF17}
N.~B{\'e}daride and Th. Fernique.
\newblock Weak local rules for octagonal tilings.
\newblock {\em Israel Journal of Mathematics}, 222:63--89, 2017.

\bibitem[BG13]{BG13}
M.~Baake and U.~Grimm.
\newblock {\em Aperiodic Order}, volume 149 of {\em Encyclopedia of mathematics
  and its applications}.
\newblock Cambridge University Press, 2013.

\bibitem[Bur88]{Bur88}
S.~E. Burkov.
\newblock Absence of weak local rules for the planar quasicrystalline tiling
  with the {$8$}-fold rotational symmetry.
\newblock {\em Communications in Mathematical Physics}, 119:667--675, 1988.

\bibitem[DB81]{DB81}
N.~G. De~Bruijn.
\newblock Algebraic theory of {P}enrose's nonperiodic tilings of the plane.
\newblock {\em Nederl. Akad. Wetensch. Indag. Math.}, 43:39--52, 1981.

\bibitem[Dev16]{sagemath}
The~Sage Developers.
\newblock {\em {S}ageMath, the {S}age {M}athematics {S}oftware {S}ystem
  ({V}ersion 7.1)}, 2016.
\newblock {\tt http://www.sagemath.org}.

\bibitem[FS18]{FS18}
Th. Fernique and M.~Sablik.
\newblock Weak colored local rules for planar tilings.
\newblock {\em Ergodic Theory and Dynamical Systems}, 2018.

\bibitem[GS86]{GS86}
B.~Gr\"{u}nbaum and G.~C. Shephard.
\newblock {\em Tilings and Patterns}.
\newblock W. H. Freeman \& Co., New York, NY, USA, 1986.

\bibitem[HJKW18]{HJKW18}
A.~Haynes, A.~Julien, H.~Koivusalo, and J.~Walton.
\newblock Statistics of patterns in typical cut and project sets.
\newblock {\em Ergodic Theory and Dynamical Systems}, pages 1--23, 2018.

\bibitem[HKSW16]{HKWS16}
A.~Haynes, H.~Koivusalo, L.~Sadun, and J.~Walton.
\newblock Gaps problems and frequencies of patches in cut and project sets.
\newblock {\em Mathematical Proceedings of the Cambridge Philosophical
  Society}, 161:65--85, 2016.

\bibitem[HP94]{HP94}
W.~V.~D. Hodge and D.~Pedoe.
\newblock {\em Methods of Algebraic Geometry}, volume~2.
\newblock Cambridge University Press, 1994.

\bibitem[Jul10]{Jul10}
A.~Julien.
\newblock Complexity and cohomology for cut-and-projection tilings.
\newblock {\em Ergodic Theory and Dynamical Systems}, 30:489–523, 2010.

\bibitem[Kat88]{Kat88}
A.~Katz.
\newblock Theory of matching rules for the 3-dimensional {P}enrose tilings.
\newblock {\em Communications in Mathematical Physics}, 118:263--288, 1988.

\bibitem[Kat95]{Kat95}
A.~Katz.
\newblock {\em Beyond Quasicrystals: Les Houches, March 7--18, 1994}, chapter
  Matching rules and quasiperiodicity: the octagonal tilings, pages 141--189.
\newblock Springer Berlin Heidelberg, Berlin, Heidelberg, 1995.

\bibitem[KP87]{KP87}
M.~Kleman and A.~Pavlovitch.
\newblock Generalised 2d {P}enrose tilings: structural properties.
\newblock {\em Journal of Physics A: Mathematical and General}, 20:687--702,
  1987.

\bibitem[Le95]{Le95}
T.~Q.~T. Le.
\newblock Local rules for pentagonal quasi-crystals.
\newblock {\em Discrete \& Computational Geometry}, 14:31--70, 1995.

\bibitem[Le97]{Le97}
T.~Q.~T. Le.
\newblock Local rules for quasiperiodic tilings.
\newblock In {\em The mathematics of long-range aperiodic order ({W}aterloo,
  {ON}, 1997)}, volume 489 of {\em NATO Adv. Sci. Inst. Ser. C Math. Phys.
  Sci.}, pages 331--366. Kluwer Acad. Publ., Dordrecht, 1997.

\bibitem[Lev88]{Lev88}
L.~S. Levitov.
\newblock Local rules for quasicrystals.
\newblock {\em Communications in Mathematical Physics}, 119:627--666, 1988.

\bibitem[LP95]{LP95}
T.~Q.~T. Le and S.~Piunikhin.
\newblock Local rules for multi-dimensional quasicrystals.
\newblock {\em Differential Geometry and its Applications}, 5:10--31, 1995.

\bibitem[LPS92]{LPS92}
T.~Q.~T. Le, S.~Piunikhin, and V.~Sadov.
\newblock Local rules for quasiperiodic tilings of quadratic {$2$}-planes in
  {${\bf R}^4$}.
\newblock {\em Communications in Mathematical Physics}, 150:23--44, 1992.

\bibitem[Pen78]{Pen78}
R.~Penrose.
\newblock Pentaplexity: a class of non-periodic tilings of the plane.
\newblock {\em Eureka}, 39, 1978.

\bibitem[Sen95]{Sen95}
M.~Senechal.
\newblock {\em Quasicrystals and geometry}.
\newblock Cambridge University Press, 1995.

\bibitem[Soc90]{Soc90}
J.~E.~S. Socolar.
\newblock Weak matching rules for quasicrystals.
\newblock {\em Communications in Mathematical Physics}, 129:599--619, 1990.

\end{thebibliography}

\end{document}